\documentclass{amsart}
\usepackage{datetime}
\usepackage{amsfonts}
\usepackage{amsmath}
\usepackage{amssymb}
\usepackage{amsthm}
\usepackage{eucal}
\usepackage[usenames,dvipsnames,svgnames,table]{xcolor}

\usepackage{booktabs}
\usepackage[font=small,labelfont={bf,sf}, textfont={sf},margin=0em]{caption}
\usepackage{hyperref}
\hypersetup{colorlinks,
  filecolor=black,
  linkcolor=MidnightBlue,
  citecolor=NavyBlue,
  urlcolor=RoyalBlue,
  bookmarksopen=true}
  
\theoremstyle{plain}
\newtheorem{theorem}{Theorem}

\newtheorem{lemma}[theorem]{Lemma}

\newtheorem*{proposition}{Proposition}
\newtheorem*{main}{Main Theorem}

\theoremstyle{definition}
\newtheorem*{notation}{Assumptions and Notation}
\newtheorem{remark}{Remark}

\newcommand{\mb}{\mathbb}
\newcommand{\mc}{\mathcal}
\newcommand{\mr}{\mathrm}

\newcommand{\vp}{\varphi}
\newcommand{\ve}{\varepsilon}
\newcommand{\cv}{\nabla}

\newcommand{\icpn}{\int_{\mb{CP}^N}}
\newcommand{\fs}{_{\mr{FS}}}

\newcommand{\lb}{\langle}
\newcommand{\rb}{\rangle}
\newcommand{\dv}{\mr{div}}
\newcommand{\D}[1]{
	\ifnum1=#1
		{\frac{\mr d}{\mr d s}\Big|_{s=0}}
	\else
		{\frac{\mr d^{#1}}{\mr d s^{#1}}\Big|_{s=0}}
	\fi}

\DeclareMathOperator{\rc}{Rc}
\DeclareMathOperator{\Rm}{Rm}
\DeclareMathOperator{\tr}{tr}

\newcounter{mnotecount}[section]
\let\oldmarginpar\marginpar
\renewcommand\marginpar[1]{\-\oldmarginpar[\raggedleft\footnotesize #1]
{\raggedright\footnotesize #1}}

\begin{document}

\title{Dynamic instability of $\mb{CP}^N$ under Ricci flow}

\author{Dan Knopf}
\address[Dan Knopf]{University of Texas at Austin}
\email{danknopf@math.utexas.edu}
\urladdr{http://www.ma.utexas.edu/users/danknopf}

\author{Nata\v sa \v Se\v sum}
\address[Nata\v sa \v Se\v sum]{Rutgers University}
\email{natasas@math.rutgers.edu}
\urladdr{http://www.math.rutgers.edu/$\sim$natasas/}

\begin{abstract}
The intent of this short note is to provide context for and an independent proof of the discovery
of Klaus Kr\"oncke~\cite{Kro13} that complex projective space with its canonical Fubini--Study
metric is dynamically unstable under Ricci flow in all complex dimensions $N\geq2$. The
unstable perturbation is not K\"ahler. This provides a counterexample to a well known conjecture
widely attributed to Hamilton. Moreover, it shows that the expected stability of the subspace of
K\"ahler metrics under Ricci flow, another conjecture believed by several experts, needs to be
interpreted in a more nuanced way than some may have expected.
\end{abstract}


\maketitle
\setcounter{tocdepth}{1}
\tableofcontents

\section{Introduction}

Ricci solitons are stationary solutions of the Ricci flow dynamical system on the space of Riemannian metrics,
modulo diffeomorphism and scaling. As such, it is natural to investigate their stability. Several distinct but related
notions of stability appear in the literature. \textsc{(i)} One says a soliton metric $g_0$ is
\emph{dynamically stable} if for every neighborhood $\mc V$ of $g_0$ there exists a neighborhood
$\mc U\subseteq\mc V$ such that every (modified) Ricci flow solution originating in $\mc U$ remains in $\mc V$
and converges to a soliton in $\mc V$. (This is subtly different from the classical definition for continuous-time
dynamical systems for two reasons. Unless a soliton is Ricci flat, one needs to modify the flow to make the soliton
into a  \emph{bona fide} fixed point. Furthermore, stationary solutions frequently occur in families, so one typically
has to deal with the presence of a center manifold.) \textsc{(ii)} One says $g_0$ is \emph{variationally stable} if the 
second variation of an appropriate choice of Perelman's energy or entropy functionals is nonpositive at  $g_0$. 
(Note that the expander entropy was defined in~\cite{FIN05}.) \textsc{(iii)} One says that $g_0$ is
\emph{linearly  stable} if the linearization of Ricci flow at $g_0$  has  nonpositive spectrum. Variational and linear stability
are equivalent in the compact case, in the precise sense that the second-variation operator $\mc N$ defined in~\eqref{sv}
below  satisfies $2\mc N=\Delta_\ell-1/\tau$ when acting on trace-free divergence-free tensors, where $\Delta_\ell$ is the
Lichnerowicz Laplacian.\footnote{There is a subtlety in this equivalence and in proving \textsc{(iii)} $\Rightarrow$ 
\textsc{(i)} that should be noted. The linearization of Ricci flow is not strictly parabolic, because the flow is invariant under the
action of the infinite-dimensional diffeomorphism group. So one either has to work orthogonally to the orbit of that group,
or else fix a gauge, which converts the linearization into the  Lichnerowicz Laplacian. Such a choice is necessary if one wishes to 
prove  dynamic stability by standard semigroup methods, but can create a finite-dimensional unstable manifold within the orbit of 
the  diffeomorphism group, especially in the presence of positive curvature. See,  \emph{e.g.,} Lemma~7 and the remarks
that follow in~\cite{Kno09}.}

Techniques to prove the implication \textsc{(iii)} $\Rightarrow$ \textsc{(i)} for flat metrics were first developed
by one of the authors and collaborators.~\cite{GIK02}. Generalizing these, the other author proved for Ricci-flat
metrics that \textsc{(i)} $\Rightarrow$ \textsc{(iii)}, and that the  converse holds if $g_0$ is integrable~\cite{Ses06}.
Haslhofer--M\"uller removed the integrability assumption, replacing it with the slightly stronger hypothesis that
$g_0$ is a local maximizer of Perelman's $\lambda$-functional \cite{HM14}. This must be verified directly and
does not follow from \textsc{(ii)}, because these metrics are only weakly variationally stable. In any case, the implication \textsc{(i)} $\Rightarrow$ \textsc{(ii)} is clear for all compact solitons.

There are many other applications of stability theory to Ricci flow, too numerous to survey thoroughly
here. An interested reader may consult the following (not complete!) list of examples, ordered by publication year:
\cite{Ye93}, \cite{DWW05},  \cite{DWW07}, \cite{SSS08}, \cite{Kno09}, \cite{LY10},  \cite{SSS11},  \cite{Has12}, 
\cite{Wu13}, \cite{Bam14},  and \cite{WW16}.

\medskip
Kr\"oncke extended the techniques of Haslhofer--M\"uller~\cite{HM14} to study stability at general
compact Einstein manifolds. Our interest here is narrower. In his investigation, Kr\"oncke discovered 
(Corollary~1.8) the unexpected fact that complex projective space with its canonical Fubini--Study metric
is dynamically unstable. We note that Kr\"oncke's work~\cite{Kro13} was accepted for publication after the
original version of this paper was completed.  Nonetheless, we believe his instability result deserves
independent attention because of its considerable broad interest, for at least three reasons:
\smallskip

\textsc{(a)} The result negatively answers a well-known conjecture widely attributed to Hamilton.
See, \emph{e.g.,} the discussion of his conjecture in the introduction to~\cite{CZ12}, as well as
Hamilton's own analysis of the behavior of $\mb{CP}^2$ in Section~10 of~\cite{Ham95}.

\textsc{(b)} The result shows that the conjecture of experts that the subspace of K\"ahler metrics is an
attractor for Ricci flow involves more complicated dynamical behavior than some may have expected. It is well
known that $(\mb{CP}^N,g_{\mr{FS}})$ is stable under K\"ahler perturbations~\cite{TZ13}. And it is easy to
see that its unstable perturbation is not K\"ahler. However, if a flow originating at this perturbation
asymptotically approaches the subspace of K\"ahler metrics, monotonicity of Perelman's shrinker
entropy $\nu$ implies that it cannot do so at the nearest K\"ahler candidate $(\mb{CP}^N,g\fs)$.
Rather, if it converges to a K\"ahler singularity model, that metric must be sufficiently far away.

\textsc{(c)} Finally, in real dimension $n=4$, dynamic instability of $\mb{CP}^2$ shows that only the first four metrics
in the  list of conjectured stable singularity models, $\mc S^4 > \mc S^3\times\mb R > \mc S^2\times\mb R^2 > \mc L^2_{-1} > \mb{CP}^2$,
ordered by the central density $\Theta$ introduced in~\cite{CHI04}, are candidates to be generic
singularity models. Here, $\mc L^2_{-1}$ denotes the ``blowdown soliton'' discovered in~\cite{FIK03}.
At least in this dimension, it is reasonable to conjecture that solutions originating at the unstable
perturbation of $\mb{CP}^2$ become singular in finite time by crushing the distinguished fiber
$\mb{CP}^1\subset\mb{CP}^2$, and converge after parabolic rescaling to $\mc L^2_{-1}$.
For further context, see the discussion and related results in~\cite{IKS17}.

\medskip
Because of these important consequences, we believe that it is valuable to produce an independent proof of  Kr\"oncke's discovery, using related but distinct methods. That is the purpose of this short
note, in which we reprove the following:

\begin{main}
For all complex dimensions $N\geq2$, complex projective space with its canonical Fubini--Study metric,
$(\mb{CP}^N,g\fs)$, is not a local maximum of Perelman's shrinker entropy $\nu$. Consequently,
it is dynamically unstable under Ricci flow. The unstable perturbation is conformal but not K\"ahler.
\end{main}

\section{Perelman's entropy}

Perelman's entropy functional $\mc W$, introduced in~\cite{Per02}, is defined on a closed Riemannian
manifold $(\mc M^n,g)$ by
\[
\mc W(g,f,\tau)=\int_{\mc M^n}\Big\{\tau\big(R+|\cv f|^2\big)+(f-n)\Big\}
(4\pi\tau)^{-\frac{n}{2}}e^{-f}\,\mr dV.
\]
Suitably minimizing $\mc W$ yields the shrinker entropy,
\[
\nu[g]=\inf\Big\{\mc W(g,f,\tau)\colon f\in C^\infty(\mc M^n),\,\tau>0,\,(4\pi\tau)^{-n/2}\int e^{-f}\,\mr dV=1\Big\}.
\]
Along a smooth variation $g(s)$ such that ${\D1}g(s)=h$, Perelman showed that
\[
{\D1}\nu\big[g(s)\big]=(4\pi\tau)^{-n/2}\int_{\mc M^n}\big\lb\tfrac12g-\tau(\rc+\cv^2 f),\,h\big\rb\,e^{-f}\,\mr dV.
\]
From this, he obtained the beautiful result that $\nu$ is nondecreasing along any compact solution of
Ricci flow --- in fact, strictly increasing except on gradient shrinking solitons normalized so that
\begin{equation}	\label{soliton}
\rc+\cv^2 f=\frac{1}{2\tau}g.
\end{equation}

As originally observed in~\cite{CHI04} --- see~~\cite{CZ12} for details, along with a complete derivation
of the strictly more complicated formula that holds at a nontrivial shrinking soliton --- if $g(s)$ is a smooth family
of metrics on $\mc M^n$ such that $g=g(0)$ is Einstein, then the second variation of $\nu$ at $s=0$ is given by
\begin{equation}	\label{SecondVariation}
{\D2} \nu\big[g(s)\big]=\frac{\tau}{V}\int_{\mc M^n} \lb \mc N(h),h\rb\,\mr dV,
\end{equation}
where $V=\mr{Vol}(\mc M^n,g)$, and
\begin{equation}	\label{sv}
\mc N(h)=\frac12\Delta h + \Rm(h,*)+\dv^*\dv\,h +\frac12\cv^2 v_h -\frac{\bar H}{2n\tau}g.
\end{equation}
Here $\bar H = (\int_{\mc M^n}H\,\mr dV)/V$ is the mean of $H=\tr_g h$,
and $v_h$ at $s=0$ is the unique solution of
\begin{equation}	\label{v_at_s=0}
\Big(\Delta+\frac{1}{2\tau}\Big)v_h=\cv^k\cv^\ell h_{\ell k}\qquad\mbox{satisfying}\qquad
\int_{\mc M^n}v_h\,\mr dV=0.
\end{equation}
In components, we write~\eqref{sv} as
\begin{multline}		\label{N}
\big(\mc N(h)\big)_{ij}=\frac12(\Delta h)_{ij}+R_{kij}^\ell g^{km}h_{m\ell}
-\frac12 g^{k\ell}(\cv_i\cv_\ell h_{kj}+\cv_j\cv_\ell h_{ki})\\
+\frac12\cv_i\cv_j v_h -\frac{\bar H}{2n\tau}g_{ij},
\end{multline}
where our index convention is $R_{ijk}^\ell=\partial_i\Gamma_{jk}^\ell-\partial_j\Gamma_{ik}^\ell%
+\Gamma_{im}^\ell\Gamma_{jk}^m-\Gamma_{jm}^\ell\Gamma_{ik}^m$. For later use, we introduce the variant
\[
\tilde{\mc N}=\frac12\Delta h + \Rm(h,*)+\dv^*\dv\,h +\frac12\cv^2 v_h
\qquad\Big(=\mc N + \frac{\bar H}{2n\tau}g\Big).
\]

It is a classical fact~\cite{Bes87} that at any compact Einstein manifold other than the standard sphere,
the space $C^\infty(S_2(T^*\mc M))$ of smooth sections of the bundle of symmetric covariant
$2$-tensors admits an orthogonal decomposition
\[
C^\infty(S_2(T^*\mc M))=\mr{im}(\dv^*) \oplus \mc C \oplus
\big(\ker(\dv)\cap\ker(\mr{tr})\big),
\]
where $\mc C$ is the space of infinitesimal conformal transformations, $C^\infty(\mc M^n)* g$.
This decomposition is also orthogonal with respect to the second variation of $\nu$ (see
Theorem~1.1 of~\cite{CH15}).

It is well known (see Example~2.3 of~\cite{CHI04} or Theorem~1.4 of~\cite{CH15}) that
$(\mb{CP}^N,g\fs)$ is neutrally variationally stable. The neutral direction is attained by
$h=\vp g\in\mc C$, where $\vp$ belongs to the first nontrivial eigenspace of the Laplacian,
\begin{equation}	\label{eigenfunction}
\Big(\Delta_{\mr{FS}}+\frac{1}{\tau}\Big)\vp=0,\qquad\icpn\vp\,\mr dV_{\mr{FS}}=0,
\end{equation}
where the subscripts indicate that the Laplacian and volume form are those of the Fubini-Study metric
$g_{\mr{FS}}$. Note that by~\eqref{soliton}, its Einstein constant is $1/(2\tau)$.

\section{Variational formulas}

In this section, we calculate the third variation of Perelman's shrinker entropy at the Fubini--Study metric.
We begin by recalling some classical first-variation formulas~\cite{Bes87}.

\begin{proposition}Let $g(s)$ be a smooth one-parameter variation of $g=g(0)$
such that ${\D1} g(s)=h$, and set $H=\tr_g h$. Then one has:
\begin{align}
\label{inverse-h}
{\D1}g^{ij}&=-h^{ij},\\
\label{Christoffel-h}
{\D1}\Gamma_{ij}^k&=\frac12\big(\cv_i h_j^k+\cv_j h_i^k-\cv^k h_{ij}\big),\\
\label{Riemann-h}
{\D1}R_{ijk}^\ell &=\frac12
\left\{\begin{array}[c]{c}
\cv_i\cv_k h_j^\ell-\cv_i\cv^\ell h_{jk}-\cv_j\cv_k h_i^\ell\\
\qquad+\cv_j\cv^\ell h_{ik}+R_{ijm}^\ell h_k^m-R_{ijk}^m h_m^\ell
\end{array}\right\},\\
\label{R-h}
{\D1}R&=-\Delta H +\dv(\dv\,h)-\lb\rc,h\rb,\\
\label{VolumeForm-h}
{\D1}\mr dV &= \frac12 H\,\mr dV.
\end{align}
\end{proposition} 
\smallskip

\begin{notation}
For simplicity, we adopt the following conventions.
$(\mb{CP}^N,g\fs)$ denotes complex projective space with its canonical Fubini--Study metric.
Then one has $\rc[g\fs]=(2\tau)^{-1} g\fs$. We denote its real dimension
by $n=\dim_{\mb R}\mb{CP}^N=2\dim_{\mb C}\mb{CP}^N=2N$.
By compactness, we may take $g(s)$ to be the smooth family
\begin{equation}
\label{eq-variation-vp}
g(s)=g\fs+sh, \qquad s\in(-\ve,\ve),
\end{equation}
where
\[
h=\vp\,g\fs,
\]
with $\vp$ the unique solution of~\eqref{eigenfunction}. Note that in the variation~\eqref{eq-variation-vp} that we consider,
$\vp$ is the fixed function defined in~\eqref{eigenfunction}, and hence is independent of $s$.
Below, to avoid notational prolixity, we write $g$ for $g(s)$ at arbitrary $s\in(-\ve,\ve)$. Formulas should be
assumed to hold at any $s$ unless explicitly stated otherwise or decorated with $\big|_{s=0}$, in which case everything
in sight is to be evaluated at $(\mb{CP}^N,\,g(0)=g\fs)$. If $A(s)$ is any smooth one-parameter
family of tensor fields depending on $s\in(-\ve,\ve)$, we write
\[
\frac{\mr d}{\mr ds}A
\]for the derivative evaluated at arbitrary $s$, and
\[
A^\prime={\D1}A
\]
for the derivative evaluated at $s = 0$. We adopt similar notation for higher-order derivatives.
Thus, as noted above, our selected variation has the property
that $\vp^\prime=0$, \emph{i.e.}, $0=h^\prime=h^{\prime\prime}=\cdots$,
a fact that we use frequently below. Finally, having fixed $h$, we simply write
$\mc N=\mc N(h)$ and $v=v_h$ where no confusion will result.
\end{notation}

\begin{remark}
We note that to prove dynamic instability of $\mb{CP}^N$, it suffices to exhibit one unstable variation.
Our choice with $\vp^\prime=0$ matches that used by Cao--Zhu~\cite{CZ12}, whose results we use below.
Kr\"oncke makes a different choice in the proof of Proposition~9.1 of~\cite{Kro13}, taking $g(t)=\big(1+t\,v(t)\big)g\fs$, 
where $v(t)=\vp/(1+t\vp)$. Other choices (differing to second order or above) are certainly possible.
\end{remark}

\begin{remark}
Our simple choice of variation, with $\vp^\prime=0$, means that we do not have the freedom
to force equation~\eqref{eigenfunction} to hold for all $s$ in an open set. Consequently, we do
not differentiate that equation in what follows.
\end{remark}

We now establish various identities needed to compute
\[
{\D3} \nu\big[g(s)\big].
\]
Even though many of these identities are well known to experts, we derive them here to
keep this note as transparent and self-contained as possible.


\begin{lemma} If $g$ and $h=\vp\,g\fs$ are as above, then using $\delta$ to denote the
Kronecker delta function, one has:
\begin{align}
\label{inverse}
{\D1}g^{ij}&=-\vp\,g^{ij},\\
\label{Christoffel}
{\D1}\Gamma_{ij}^k&=\frac12\big(\cv_i\vp\,\delta_j^k+\cv_j\vp\,\delta_i^k-\cv^k\vp\,g_{ij}\big),\\
\label{Riemann}
{\D1}R_{ijk}^\ell&=\frac12\big(\cv_i\cv_k\vp\,\delta_j^\ell-\cv_j\cv_k\vp\,\delta_i^\ell
+\cv_j\cv^\ell\vp\,g_{ik}-\cv_i\cv^\ell\vp\,g_{jk}\big),\\
\label{R}
{\D1}R&=-(n-1)\Delta\vp-\frac{n}{2\tau}\vp,\\
\label{DdV}
{\D1}\mr dV &= \frac{n}{2}\vp\,\mr dV.
\end{align}
\end{lemma}

\begin{proof}
Formulas~\eqref{inverse}--\eqref{DdV} follow from~\eqref{inverse-h}--\eqref{VolumeForm-h}
by direct substitution.
\end{proof}

\begin{lemma}	\label{SpecialVariations}
If $g$ and $h=\vp\,g_{\mr{FS}}$ are as above, then one has:
\begin{align}
\label{Dtau}
{\D1}\tau&=0,\\
\label{DV}
{\D1} V &=0,\\
\label{DbarH}
{\D1}\bar H &=\frac{n(n-2)}{2V}\|\vp\|^2,
\end{align}
where $\|\vp\|^2 = \icpn \vp^2\mr dV$.
\end{lemma}

\begin{proof}
To establish~\eqref{Dtau}, we recall that by Lemma~2.4 of~\cite{CZ12}, using the fact that our metric is Einstein with
$\rc = (2\tau)^{-1} g_{FS}$ (and hence that we can take $f = 0$ in Lemma 2.4 of \cite{CZ12}), we have
\[
\tau^\prime = \tau\frac{\int_{\mc M^n}\lb\rc,h\rb\,\mr dV}{\int_{\mc M^n}R\,\mr dV}.
\]
Furthermore, at $s=0$, equation~\eqref{eigenfunction} implies that
\[
\icpn \lb\rc,h\rb\,\mr dV=\frac{1}{2\tau}\icpn H\,\mr dV=\frac{n}{2\tau}\icpn\vp\,\mr dV=0.
\]
Equation~\eqref{Dtau} follows.


Equation~\eqref{DV} follows from~\eqref{DdV} and~\eqref{eigenfunction}, because at $s=0$,
\[
V^\prime = \icpn (\mr dV)^\prime = \frac{n}{2}\icpn\vp\,\mr dV=0.
\]

Finally, equation~\eqref{DbarH} follows from~\eqref{DV} and the computation
\begin{align*}
{\D1}\icpn g^{ij}h_{ij}\,\mr dV
	&=\icpn (g^{ij})^\prime h_{ij}\,\mr dV + \icpn g^{ij}h_{ij}\,(\mr dV)^\prime\\
	&=\icpn\big(-\vp H +H\,\frac{n}{2}\vp\big)\,\mr dV\\
	&=\Big(-n+\frac{n^2}{2}\Big)\icpn\vp^2\,\mr dV.
\end{align*}
\end{proof}

We next compute derivatives of the Laplacian.

\begin{lemma}	\label{DLaplacian_lemma}
If $g$ and $h=\vp\,g_{\mr{FS}}$ are as above, then (at $s = 0$) we have
\begin{equation}	\label{DLaplacian}
\Delta^\prime=-\vp\Delta+\frac{n-2}{2}\,\cv_{\cv\vp}.
\end{equation}
\end{lemma}

\begin{proof}
For any smooth (possibly $s$-dependent) function $u$, using
formulas~\eqref{inverse} and~\eqref{Christoffel} and collecting terms yields
\begin{align*}
{\D1}\big\{g^{ij}(\partial_i\partial_j u-\Gamma_{ij}^k\partial_k u)\big\}
&=\Delta u^\prime-\vp\Delta u-g^{ij}(\Gamma_{ij}^k)^\prime\cv_k u\\
&=\Delta u^\prime-\vp\Delta u+\frac{n-2}{2}\lb\cv\vp,\cv u\rb,
\end{align*}
whence the result follows. (We note that for $u = \vp$, we have $\vp' = 0$.)
\end{proof}

\begin{lemma}
\label{DDLaplacian}
If $g$ and $h = \vp g_{FS}$ are as above, then (at $s = 0$) we have
\[\Delta'' = 2\vp^2 \Delta.\]
\end{lemma}

\begin{proof}
If $u$ is a smooth function that is independent of $s$, then
\[
(\Delta u)^{\prime\prime}=(g^{ij})^{\prime\prime}\cv_i\cv_j u+2(g^{ij})^\prime(\cv_i\cv_j u)^\prime
+g^{ij}(\cv_i\cv_j u)^{\prime\prime}.
\]
Using the Neumann series for $(g\fs+s\vp g\fs)^{-1}$, one sees easily that $(g^{ij})^{\prime\prime}=2\vp^2g^{ij}$.
Equation~\eqref{inverse} gives $(g^{-1})^\prime=-\vp g^{-1}$, and as in the proof of Lemma~\ref{DLaplacian_lemma}, one has
\[
g^{ij}(\cv_i\cv_j u)^\prime=\frac{n-2}{2}\lb\cv\vp,\,\cv u\rb.
\]
Using the fact that $u^\prime=0$, one has
\[
(\cv_i\cv_j u)^{\prime\prime}=-(\Gamma_{ij}^k)^{\prime\prime}\cv_k u,
\]
where by part~(1) of Lemma~28 of~\cite{AK07}, we have
\begin{align*}
(\Gamma_{ij}^k)^{\prime\prime}&=-h_\ell^k(\cv_i h_j^\ell+\cv_j h_i^\ell-\cv^\ell h_{ij})\\
&=-\vp(\cv_i\vp\,\delta_j^k+\cv_j\vp\,\delta_i^k-\cv^k\vp\,g_{ij}).
\end{align*}
Putting these together, we obtain
\[
(\Delta u)^{\prime\prime}=2\vp^2\Delta u-2\vp\,\frac{n-2}{2}\lb\cv\vp,\,\cv u\rb+\vp(2-n)\lb\cv\vp,\,\cv u\rb
=2\vp^2\Delta u,
\]
which finishes the proof.
\end{proof}

\begin{lemma}	\label{integrand}
If $g$ and $h=\vp\,g_{\mr{FS}}$ are as above, then at $s=0$, $v=2\vp$,  and $\tilde{\mc N}$ vanishes pointwise.
\end{lemma}

\begin{proof}
To prove the first claim, we note that $\cv^k\cv^\ell h_{\ell k}=\Delta\vp=-\vp/\tau$ by~\eqref{eigenfunction}, and
hence by~\eqref{v_at_s=0} we have
\[\left(\Delta  +\frac{1}{2\tau}\right)\, v = - \frac{\vp}{\tau}, \qquad \mbox{with} \,\,\,\, \icpn v\mr dV = 0.\]
On the other hand, by \eqref{eigenfunction}, we have
\[\left(\Delta + \frac{1}{2\tau}\right)\, 2\vp = -\frac{\vp}{\tau}, \qquad \mbox{with} \,\,\,\, \icpn 2\vp \mr dV = 0.\]
By uniqueness, we have $v = 2\vp$.

To prove the second claim, one uses the identity $R_{ij}\big|_{s=0} = (2\tau)^{-1} g_{ij}$,
without differentiating, along with~\eqref{eigenfunction} and the first claim to obtain
\begin{align*}
\tilde{\mc N}_{ij}&=\frac12\Delta\vp\,g_{ij}+\vp R_{ij}-\cv_i\cv_j\vp+\frac12\cv_i\cv_j v\\
&= \frac 12\Big(\Delta \varphi + \frac{1}{\tau} \varphi\Big)\, g_{ij} + \cv_i\cv_j\Big(\frac12v-\vp\Big)=0.
\end{align*}
\end{proof}

We now compute the third variation explicitly.

\begin{lemma}	\label{Reduction}
If $g$ and $h=\vp\,g_{\mr{FS}}$ are as above, then
 \[
{\D3} \nu\big[g(s)\big]=
-\frac{\tau}{(4\pi\tau)^{n/2}}\int_{\mc M^n}
\Big\lb h,\,
\rc^{\prime\prime}+(\cv^2f)^{\prime\prime}-\Big(\frac{1}{2\tau}g\Big)^{\prime\prime}\Big\rb\,\mr dV,
\]
where the three second derivatives on the right-hand side are given pointwise by~\eqref{eq-2nd-ric},
\eqref{eq-2nd-hess-f}, and \eqref{eq-2nd-tau}, respectively.
\end{lemma}

\begin{proof}
By Lemma~2.2 of~\cite{CZ12}, at arbitrary $s\in(-\ve,\ve)$, one has
\[
\begin{split}
\frac{\mr d}{\mr ds}\nu(s) &=-\big\{\tau(4\pi\tau)^{-n/2}\big\}\icpn
\big\langle h, \,\rc+\cv^2  f-\frac{1}{2\tau}g \big\rangle\,e^{-f}\mr dV \\
&= - \big\{\tau(4\pi\tau)^{-n/2}\big\}\icpn
\big(g^{ip}g^{jq}h_{pq}\big) \Big(R_{ij} + \cv_i\cv_j f - \frac{1}{2\tau} g_{ij}\big)\, \big(e^{-f} \mr dV\big),
\end{split}
\]
which we write schematically as $A\icpn B* C* D$.
A simple calculus exercise shows that
\begin{align*}
\nu^{\prime\prime\prime}(s)&=
A^{\prime\prime}\icpn B*C*D+A\icpn B^{\prime\prime}*C*D
+A\icpn B*C^{\prime\prime}*D+A\icpn B*C*D^{\prime\prime}\\
&\quad+2 A^\prime\icpn B^\prime* C* D+2 A^\prime\icpn B*C^\prime* D+2A^\prime\icpn B*C*D^\prime\\
&\quad+2A\icpn B^\prime*C^\prime* D
+2A\icpn B^\prime*C*D^\prime+2A\icpn B*C^\prime*D^\prime.
\end{align*}
Because $(\mb{CP}^N,\,g\fs)$ is Einstein, we have $C\big|_{s=0}=0$.
By~\eqref{Dtau}, $A^\prime=0$. 
Thus the formula above reduces to
\[
\nu^{\prime\prime\prime}=A\icpn B*C^{\prime\prime}*D
+2A\icpn B^\prime*C^\prime*D + 2A\icpn B*C^\prime*D^\prime.
\]
Using Lemma~2.3 of~\cite{CZ12}, and~\eqref{Dtau}, we obtain
\begin{align*}
C^\prime&=(\rc+\cv^2 f)^\prime-\frac{1}{2\tau}h\\
&=-\frac12\Delta h -\Rm(h,\cdot)-\dv^*\dv\,h-\cv^2v\\
&=-\tilde{\mc N}|_{s=0}.
\end{align*}
By Lemma~\ref{integrand}, we have $C^\prime=0$, and hence
\begin{equation}
\label{eq-3rd-var}
\nu^{\prime\prime\prime}(s)=A\icpn B* C^{\prime\prime}* D.
\end{equation}

There are three terms in $C^{\prime\prime}$. To compute the first, we apply part~(3) of Lemma~28 of~\cite{AK07}
to our conformal variation, obtaining
\begin{equation}
\label{eq-2nd-ric}
R^{\prime\prime}_{ij}=\Big(\vp\Delta\vp-\frac{n-2}{2}|\cv\vp|^2\Big)g_{ij}
+(n-2)\Big(\vp\cv_i\cv_j\vp+\cv_i\vp\cv_j\vp\Big),
\end{equation}
where the right-hand side is computed with respect to the metric data of $g\fs$.

To compute the second term, we begin by observing that
\[
(\cv_i\cv_j f)^{\prime\prime}=\cv_i\cv_j f^{\prime\prime}-2(\Gamma_{ij}^k)^\prime\cv_k f^\prime
-(\Gamma_{ij}^k)^{\prime\prime}\cv_k f,
\]
As noted above, the function $v(s)$ defined by
\[
v:=-2f^\prime+H-2\frac{\tau^\prime}{\tau}(f-\nu)
\]
in the discussion on page~9 of~\cite{CZ12} is given by~\eqref{v_at_s=0} at $s=0$.
Moreover, by~\eqref{Dtau},
we have
\begin{equation}	\label{v_identity}
v\big|_{s=0}= -2f' + H.
\end{equation}
Then equation~\eqref{v_identity} and Lemma \ref{integrand} imply that 
\[
f^\prime = \frac{n-2}{2}\vp.
\]
Now the fact that $g\fs$ is Einstein implies that $f\big|_{s=0}=0$, and so by \eqref{Christoffel} we get
\begin{equation}	\label{eq-2nd-hess-f}
\begin{split}
(\nabla_i\nabla_j f)''|_{s=0} &= \nabla_i\nabla_j f'' + \langle\vp, \nabla f'\rangle g_{ij} - \nabla_i\vp\nabla_j f' - \nabla_j\vp\nabla_i f' \\
& =\cv_i\cv_j f^{\prime\prime}+\frac{n-2}{2}|\cv\vp|^2g_{ij}
-(n-2)\cv_i\vp\cv_j\vp.
\end{split}
\end{equation}

Similarly, using~\eqref{Dtau} and the fact that $g^{\prime\prime}=0$, we compute that the final term is
\begin{equation}
\label{eq-2nd-tau}
\Big(\frac{1}{2\tau}g\Big)^{\prime\prime}=-\frac{\tau^{\prime\prime}}{2\tau^2}g.
\end{equation}
The result follows, because $D\big|_{s=0}=\big(e^{-f} \mr dV\big)\big|_{s=0}=\mr dV$.
\end{proof}

To calculate $\nu^{\prime\prime\prime}$, we will use the following identity.

\begin{lemma}
\label{lem-Lapf''}
If $g$ and $h=\vp\,g_{\mr{FS}}$ are as above, then
\[
\left(\Delta  +\frac{1}{2\tau}\right)f^{\prime\prime}=
-n\vp\Delta\vp-\frac{n}{2\tau}\vp^2-\frac{n\tau^{\prime\prime}}{4\tau^2}.
\]
\end{lemma}

\begin{proof}
Recall that since $f=f(s)$ is the minimizer of $\nu=\nu(s)$ for every $s$, it satisfies the following elliptic equation:
\[\tau(-2\Delta f + |\nabla f|^2 - R) - f + \nu = 0.\]
At $s = 0$, we have $\tau^\prime=0$ by~\eqref{Dtau} and $\nu^{\prime\prime}= 0$ by our choice of variation.
(Indeed, at $s=0$, we have $\nu^{\prime\prime} = A\icpn B*C^\prime*D=0$, as seen in the proof of Lemma~\ref{Reduction}.)
Because $f\big|_{s=0} = 0$, it follows that at $s=0$,
\begin{equation}	\label{eq-ident-min}
2\tau (\Delta f)^{\prime\prime} + f^{\prime\prime} = \tau \Big\{(\big|\nabla f|^2\big)^{\prime\prime} - R^{\prime\prime}\Big\}
- R\tau^{\prime\prime}.
\end{equation}
Again using $f\big|_{s=0} = 0$, $(g^{-1})^{\prime\prime}=2\vp^2 g^{-1}$ (see Lemma~\ref{DDLaplacian}),
and $f' = \big((n-2)/2\big)\vp$ (see Lemma~\ref{Reduction}), we get
\begin{align}	\label{eq-nabla-f''}
(|\cv f|^2)^{\prime\prime}&=2\vp^2|\cv f|^2+4\lb\cv f,\cv f^\prime\rb+2|\cv f^\prime|^2 \notag\\
&= \frac{(n-2)^2}{2} |\cv\vp|^2.
\end{align}
Again using $f\big|_{s=0}=0$ and $(\cv_i\cv_j f)^\prime=\cv_i\cv_j f^\prime=\big((n-2)/2\big)\cv_i\cv_j\vp$ along 
with~\eqref{eq-2nd-hess-f}, we compute that
\begin{align}
(\Delta f)^{\prime\prime}&
=(g^{ij})^{\prime\prime}\cv_i\cv_j f+ 2(g^{ij})^\prime(\cv_i\cv_j f)^\prime + g^{ij}(\cv_i\cv_j f)^{\prime\prime}\notag\\
&=\Delta f^{\prime\prime}-(n-2)\vp\Delta\vp+\frac{(n-2)^2}{2}|\cv\vp|^2.	\label{Delta_f_prime}
\end{align}
Finally, contracting ~\eqref{Riemann} yields
\[
R_{jk}^\prime=-\frac{n-2}{2}\cv_j\cv_k\vp-\frac12\Delta\vp\,g_{jk}.
\]
Using this with the identities $(g^{-1})^\prime=2\vp^2g^{-1}$, $R=g^{ij}R_{ij}$ and $R=n/(2\tau)$,
along with equations~\eqref{inverse} and~\eqref{eq-2nd-ric}, we compute that
\begin{align*}
R^{\prime\prime}&=(g^{ij})^{\prime\prime}R_{ij}+2(g^{ij})^\prime R_{ij}^\prime+g^{ij}R_{ij}^{\prime\prime}\\
&=2\vp^2R+2\vp\big\{\frac{n-2}{2}\Delta\vp+\frac{n}{2}\Delta\vp\big\}
+n\big(\vp\Delta\vp-\frac{n-2}{2}|\cv\vp|^2\big)+(n-2)\big(\vp\Delta\vp+|\cv\vp|^2\big)\\
&=4(n-1)\vp\Delta\vp-\frac{(n-2)^2}{2}|\cv\vp|^2+\frac{n}{\tau}\vp^2.
\end{align*}
Now combining this identify with equations~\eqref{eq-ident-min}, \eqref{eq-nabla-f''}, and~\eqref{Delta_f_prime}, we obtain
\begin{align*}
2\tau\Delta f^{\prime\prime}+f^{\prime\prime}
&=\tau\big\{2(n-2)\vp\Delta\vp-(n-2)^2|\cv\vp|^2\big\}\\
&\qquad+\tau\big\{-4(n-1)\vp\Delta\vp+(n-2)^2|\cv\vp|^2-\frac{n}{\tau}\vp^2\big\}-R\tau^{\prime\prime}\\
&=-2n\tau\vp\Delta\vp-n\vp^2-\frac{n\tau^{\prime\prime}}{2\tau},
\end{align*}
which completes the proof.
\end{proof}

\section{Proof of the main result}

\begin{theorem}	\label{AlmostMain}
If $g$ and $h=\vp\,g_{\mr{FS}}$ are as above, then
\[
{\D3}\nu\big[g(s)\big]=\frac{n-2}{(4\pi\tau)^{n/2}}\icpn\vp^3\,\mr dV.
\]
\end{theorem}

\begin{proof}
Using Lemma~\ref{Reduction}, equations~\eqref{eq-2nd-ric}, \eqref{eq-2nd-hess-f}, and \eqref{eq-2nd-tau},
and the fact that $h=\vp g$, we compute at $s=0$, where $e^{-f}=1$, that
\begin{align*}
\frac{\mr d^3}{\mr ds^3}\Big|_{s=0}\nu\big[g(s)\big]&=
-\frac{\tau}{(4\pi\tau)^{n/2}}
\icpn\Big\{n\vp\big(\vp\Delta\vp-\frac{n-2}{2}|\cv\vp|^2\big)+(n-2)\big(\vp^2\Delta\vp+|\cv\vp|^2\big)\Big\}\,\mr dV\\
&\quad-\frac{\tau}{(4\pi\tau)^{n/2}}\icpn\big(\vp\Delta f^{\prime\prime}+\frac{n(n-2)}{2}\vp|\cv\vp|^2-(n-2)\vp|\cv\vp|^2\big)\,\mr dV\\
&\quad-\frac{\tau}{(4\pi\tau)^{n/2}}\Big(-\frac{n\tau^{\prime\prime}}{2\tau^2}\Big)\icpn \vp\,\mr dV\\
&=-\frac{\tau}{(4\pi\tau)^{n/2}}\Big\{2(n-1)\icpn\vp^2\Delta\vp\,\mr dV+\icpn \vp\Delta f^{\prime\prime}\,\mr dV\Big\},
\end{align*}
where we used the consequence of~\eqref{eigenfunction} that $\icpn\vp\,\mr dV=0$.
To calculate the second integral in the last line above, we use equation~\eqref{eigenfunction} and
Lemma~\ref{lem-Lapf''} to see that
\begin{align*}
\icpn\vp\Delta f^{\prime\prime}\,\mr dV
&= \icpn\vp\Big(\Delta+\frac{1}{2\tau}-\frac{1}{2\tau}\Big)f''\,\mr dV\\
&=-n\icpn\vp^2\Delta\vp\,\mr dV-\frac{n}{2\tau}\icpn\vp^3\,\mr dV-\frac{n\tau^{\prime\prime}}{4\tau^2}\icpn\vp\,\mr dV
-\frac{1}{2\tau}\icpn\vp f^{\prime\prime}\,\mr dV\\
&=-n\icpn\vp^2\Delta\vp\,\mr dV-\frac{n}{2\tau}\icpn\vp^3\,\mr dV+\frac12\icpn\Delta\vp f^{\prime\prime}\,\mr dV\\
&=-n\icpn\vp^2\Delta\vp\,\mr dV-\frac{n}{2\tau}\icpn\vp^3\,\mr dV+\frac12\icpn\vp\Delta f^{\prime\prime}\,\mr dV,
\end{align*}
where in the last step, we used the fact that the Laplacian is self-adjoint. Thus by using~\eqref{eigenfunction}
to evaluate $\Delta\vp$, we conclude that
\begin{align*}
\frac{\mr d}{\mr ds}\Big|_{s=0}\nu\big[g(s)\big]&=\frac{\tau}{(4\pi\tau)^{n/2}}
\Big\{2\icpn\vp^2\Delta\vp\,\mr dV+\frac{n}{\tau}\icpn\vp^3\,\mr dV\Big\}\\
&=\frac{n-2}{(4\pi\tau)^{n/2}}\icpn\vp^3\,\mr dV.
\end{align*}
\end{proof}
\medskip

Our Main Theorem now follows from Theorem~\ref{AlmostMain} and the following observation:
\begin{lemma}
If $N>1$, there is a solution $\vp$  of~\eqref{eigenfunction} such that
\[
\icpn\vp^3\,\mr dV\neq0.
\]
\end{lemma}

\begin{proof}
Our approach here is essentially the same as in~\cite{Kro13} and uses results from Part III-C
of~\cite{BGM71}. We include it here merely for completeness.
We denote by $\mc P_k$ the vector space of polynomials on $\mb C^{N+1}$ such that
$f(cz)=c^k\bar c^kf(z)$ and by $\mc H_k\subset\mc P_k$ the subspace of harmonic polynomials.
Then the eigenfunctions of the $k^{\mr{th}}$ eigenvalue of the Laplacian on $(\mb{CP}^N,g_{\mr{FS}})$
are the restrictions of functions in $\mc H_k$, and one has a decomposition
\begin{equation}	\label{DirectSum}
\mc P_k=\mc H_k\oplus r^2\mc P_{k-1}\quad\mbox{ for all integers }\quad k\geq1.
\end{equation}
It follows that the dimension of the first nontrivial eigenspace of the the Laplacian on $(\mb{CP}^N,g_{\mr{FS}})$
is $(N+1)^2-1=N(N+2)$.

Using the hypothesis $N>1$, one can define real-valued functions
\[
f_1(z)=z_1\bar z_2+\bar z_1 z_2,\quad
f_2(z)=z_2\bar z_3+\bar z_2 z_3,\quad
f_3(z)=z_3\bar z_1+\bar z_3 z_1.
\]
We choose $f=f_1+f_2+f_3$ and denote by $\vp$ its restriction to $\mb{CP}^N$. We will show that
this choice (not unique in general) has the desired properties. By considering isometries $z_j\mapsto-z_j$,
it is easy to verify that $\int_{\mc S^{2N-1}}f\,\mr dV=0$. One has
\[
f^3=\sum_{j=1}^3 f_j^3 + 3\sum_{j=1}^3\sum_{k\neq j} f_j f_k^2 + 6f_1f_2f_3.
\]
Consideration of the same isometries $z_j\mapsto-z_j$ similarly shows that the integrals of
all terms above vanish, except possibly the last. We expand that term, obtaining
\begin{align*}
f_1f_2f_3&=2|z_1|^2|z_2|^2|z_3|^2\\
&\quad
+|z_1|^2(z_2^2\bar z_3^2+\bar z_2^2z_3^2)
+|z_2|^2(z_1^2\bar z_3^2+\bar z_1^2z_3^2)
+|z_3|^2(z_1^2\bar z_2^2+\bar z_1^2z_2^2).
\end{align*}
Consideration of isometries $z_j\mapsto iz_j$ shows that the integrals of all terms here except the first vanish,
whereupon we obtain
\[
\int_{\mc S^{2N-1}}f^3\,\mr dV=12\int_{\mc S^{2N-1}}|z_1|^2|z_2|^2|z_3|^2\,\mr dV
=12\,\mr{Vol}\big(\mc S^{2N-1}\big).
\]

Now by~\eqref{DirectSum}, there are functions $F_k\in\mc H_k$ such that
$f^3=F_3+r^2F_2+r^4F_1+r^6F_0$, where the restriction of each $F_k$ to $\mc S^{2N+1}$ is an
eigenfunction of its Laplacian,
with $\int_{\mc S^{2N-1}}F_k\,\mr dV=0$ for $k=3,2,1$. The fact that $\int_{\mc S^{2N-1}}f^3\,\mr dV>0$
thus implies that $F_0>0$, and the same conclusion then holds for the corresponding decomposition
$\vp^3=\Phi_3+r^2\Phi_2+r^4\Phi_1+r^6\Phi_0$
into eigenvalues of the Laplacian on $\mb{CP}^N\approx\mc S^{2N+1}/\mc S^1$. The result follows.
\end{proof}


\begin{thebibliography}{DWW07}

\bibitem[AK07]{AK07}
\textsc{Angenent, Sigurd B.; Knopf, Dan}
Precise asymptotics of the Ricci flow neckpinch.
\emph{Comm.~Anal.~Geom.}~\textbf{15} (2007), no.~4, 773--844.

\bibitem[Bam14]{Bam14}
\textsc{Bamler, Richard H.}
Stability of hyperbolic manifolds with cusps under Ricci flow.
\emph{Adv.~Math.}~\textbf{263} (2014), 412--467.

\bibitem[BGM71]{BGM71}
\textsc{Berger, Marcel; Gauduchon, Paul; Mazet, Edmond.}
\emph{Le spectre d'une vari\'et\'e Riemannienne.}
Lecture Notes in Mathematics, Vol.~194, Springer-Verlag, Berlin--New York (1971). 

\bibitem[Bes87]{Bes87}
\textsc{Besse, Arthur L.}
\emph{Einstein manifolds.}
Ergebnisse der Mathematik und ihrer Grenzgebiete, \textbf{10}. Springer-Verlag, Berlin, 1987.

\bibitem[CHI04]{CHI04}
\textsc{Cao, Huai-Dong; Hamilton, Richard S.; Ilmanen, Tom.}
Gaussian densities and stability for some Ricci solitons.
\texttt{arXiv:math/0404165v1}.

\bibitem[CH15]{CH15}
\textsc{Cao, Huai-Dong; He, Chenxu.}
Linear stability of Perelman's $\nu$-entropy on symmetric spaces of compact type.
\emph{J.~Reine Angew.~Math.}~\textbf{709} (2015), 229--246.

\bibitem[CZ12]{CZ12}
\textsc{Cao, Huai-Dong; Zhu, Meng.}
On second variation of Perelman's Ricci shrinker entropy.
\emph{Math.~Ann.}~\textbf{353} (2012), no.~3, 747--763.

\bibitem[DWW05]{DWW05}
\textsc{Dai, Xianzhe; Wang, Xiaodong; Wei, Guofang.}
On the stability of Riemannian manifold with parallel spinors.
\emph{Invent.~Math.}~\textbf{161} (2005), no.~1, 151--176.

\bibitem[DWW07]{DWW07}
\textsc{Dai, Xianzhe; Wang, Xiaodong; Wei, Guofang.}
On the variational stability of K\"ahler--Einstein metrics.
\emph{Comm.~Anal.~Geom.}~\textbf{15} (2007), no.~4, 669--693.

\bibitem[FIK03]{FIK03}
\textsc{Feldman, Mikhail; Ilmanen, Tom; Knopf, Dan.}
Rotationally symmetric shrinking and expanding gradient K\"ahler--Ricci solitons.
\emph{J.~Differential Geom.}~\textbf{65} (2003), no.~2, 169--209.

\bibitem[FIN05]{FIN05}
\textsc{Feldman, Michael; Ilmanen, Tom; Ni, Lei.} 
Entropy and reduced distance for Ricci expanders.
\emph{J.~Geom.~Anal.}~\textbf{15} (2005), no.~1, 49--62.

\bibitem[GIK02]{GIK02}
\textsc{Guenther, Christine; Isenberg, James; Knopf, Dan.}
Stability of the Ricci flow at Ricci-flat metrics.
\emph{Comm.~Anal.~Geom.}~\textbf{10} (2002), no.~4, 741--777.

\bibitem[Ham95]{Ham95}
\textsc{Hamilton, Richard S.}
The formation of singularities in the Ricci flow.
\emph{Surveys in differential geometry,} Vol.~II (Cambridge, MA, 1993), 7---136,
Int.~Press, Cambridge, MA, 1995. 

\bibitem[Has12]{Has12}
\textsc{Haslhofer, Robert}
Perelman's $\lambda$-functional and the stability of Ricci-flat metrics.
\emph{Calc.~Var.~Partial Differential Equations} \textbf{45} (2012), no.~3--4, 481--504.

\bibitem[HM14]{HM14}
\textsc{Haslhofer, Robert; M\"uller, Reto.}
Dynamical stability and instability of Ricci-flat metrics.
\emph{Math.~Ann.}~\textbf{360} (2014), no.~1--2, 547--553.

\bibitem[IKS17]{IKS17}
\textsc{Isenberg, James; Knopf, Dan; \v Se\v sum, Nata\v sa.}
Non-K\"ahler Ricci flow singularities that converge to K\"ahler--Ricci  solitons.
\texttt{(arXiv:1703.\allowbreak02918)}

\bibitem[Kno09]{Kno09}
\textsc{Knopf, Dan.}
Convergence and stability of locally $\mb R^N$-invariant solutions of Ricci flow.
\emph{J.~Geom.~Anal.}~\textbf{19} (2009), no.~4, 817--846.

\bibitem[Kro13]{Kro13}
\textsc{Kr\"oncke, Klaus.}
Stability of Einstein metrics under Ricci flow.
\emph{Comm.~Anal. Geom.}~In press.
(\texttt{arXiv:\allowbreak1312.2224v2}).

\bibitem[LY10]{LY10}
\textsc{Li, Haozhao; Yin, Hao.}
On stability of the hyperbolic space form under the normalized Ricci flow. 
\emph{Int.~Math.~Res.~Not.} (2010), no.~15, 2903--2924. 

\bibitem [Per02]{Per02}
\textsc{Perelman, Grisha.}
The entropy formula for the Ricci flow and its geometric applications.
\texttt{(arXiv:math.DG/0211159)}.

\bibitem[Ses06]{Ses06}
\textsc{\v Se\v sum, Nata\v sa.}
Linear and dynamical stability of Ricci-flat metrics.
\emph{Duke Math.~J.} \textbf{133} (2006), no.~1, 1--26.

\bibitem[SSS08]{SSS08}
\textsc{Schn\"urer, Oliver C.; Schulze, Felix; Simon, Miles.}
Stability of Euclidean space under Ricci flow.
\emph{Comm.~Anal.~Geom.}~\textbf{16} (2008), no.~1, 127--158.

\bibitem[SSS11]{SSS11}
\textsc{Schn\"urer, Oliver C.; Schulze, Felix; Simon, Miles.}
Stability of hyperbolic space under Ricci flow.
\emph{Comm.~Anal.~Geom.}~\textbf{19} (2011), no.~5, 1023--1047.

\bibitem[TZ13]{TZ13}
\textsc{Tian, Gang; Zhu, Xiaohua.}
Convergence of the K\"ahler--Ricci flow on Fano manifolds.
\emph{J.~Reine Angew.~Math.}~\textbf{678} (2013), 223--245.

\bibitem[WW16]{WW16}
\textsc{Williams, Michael Bradford; Wu, Haotian.}
Dynamical stability of algebraic Ricci solitons.
\emph{J.~Reine Angew.~Math.}~\textbf{713} (2016), 225--243.

\bibitem[Wu13]{Wu13}
\textsc{Wu, Haotian.}
Stability of complex hyperbolic space under curvature-normalized Ricci flow.
\emph{Geom.~Dedicata} \textbf{164} (2013), 231--258.

\bibitem[Ye93]{Ye93}
\textsc{Ye, Rugang.}
Ricci flow, Einstein metrics and space forms. 
\emph{Trans.~Amer.~Math. Soc.}~\textbf{338} (1993), no.~2, 871--896. 

\end{thebibliography}
\end{document}